\documentclass[12pt,draft]{article}
\usepackage{amsmath, amsthm, amssymb}
\usepackage[margin=1in]{geometry}
\newtheorem{defi}{Definition}
\newtheorem{thm}{Theorem}
\newtheorem{lem}{Lemma}
\newtheorem{asu}{Assumption}

\begin{document}

\nocite{*}
\title{Enstrophy Cascade in Physical Scales \\for the 3D Navier-Stokes Equations}

\author{Keith Leitmeyer}

\maketitle

\begin{abstract}
 An enstrophy cascade is exhibited for the Navier-Stokes equations in physical scales independently of boundary conditions under physically reasonable assumptions on the flow.
\end{abstract}

\section{Introduction}

The Navier-Stokes equations $u_t+(u\cdot \nabla) u = \Delta u -\nabla p, \,\text{div }u=0$, where $u$ is the velocity vector and $p$ is the scalar pressure, model the dynamics of an incompressible,
viscous fluid.  However, the physical theory of turbulence is based mostly on experimental observations 
and heuristic arguments.  Two of the main features of Kolmogorov's K41 phenomenology are existence of the energy cascade and scale-locality of the energy flux in 3D.
The first rigorous mathematical result in favor of existence of the energy cascade 
was obtained in \cite{FMRT01} in the setting of ``stationary statistical solutions'', and the cascade
took place in the wavenumbers. A related work \cite{FJMR} presented the proofs of existence
of the enstrophy cascade and inverse energy cascade in 2D within the same mathematical framework
(see also \cite{R}). This approach concerns statistically steady-state turbulence, feeding off
a non-trivial external force. Mathematical arguments confirming the scale-locality of the flux
were presented in \cite{LF92, E05, EA09}, and \cite{CCFS08} which features
a rigorous proof of quasi-locality in the Littlewood-Paley setting. 

At least since G.I. Taylor's fundamental 1937 paper, \emph{Production and dissipation of vorticity in a turbulent fluid} (\cite{Tay37}), vortex stretching has been considered as a primary physical mechanism for
creation of small scales in turbulent flows (cf. \cite{Gr14} for an overview of current efforts to
establish rigorous mathematical framework for the vortex-stretching--anisotropic diffusion 
narrative). The tendency of a turbulent flow to self-organize in
coherent vortex structures, most notably in vortex filaments, has been well-documented in
direct numerical simulations, see, e.g., \cite{SJO91, VM94, JWSR93}.
Strong local anisotropy and quasi low-dimensionality of turbulent flows in
the vorticity description points to plausibility of existence of 3D enstrophy cascade. However,
in order to efficiently exploit depletion of the nonlinearity caused by the local anisotropy, and
in particular, by local coherence of the vorticity direction (cf. \cite{Co94}), it is necessary
to be able to formulate dynamics of the turbulent cascades directly in the physical scales
of the flow. 

A spatial multi-scale ensemble-averaging process designed to detect significant sign-fluctuations of
a density in view, at a given physical scale, was introduced in \cite{DG1}.
This was then utilized, taking the density to be the energy flux-density, to obtain a proof of existence
of energy cascade and scale-locality of the flux in physical scales of 3D turbulent flows,
in the case of free (decaying) turbulence.
A subsequent work \cite{DG2} presented a proof of existence of 3D enstrophy cascade in a cylinder in space-time under the assumptions on the geometry of the flow and
smallness of a Kraichnan-type micro-scale, as well as several technical assumptions.
One technical assumption, that the enstrophy -- localized to the macro-scale domain of interest --
was smaller than a given constant, was (as pointed by the authors) less than satisfactory, both from mathematical and physical
viewpoints.

In what follows, we make modifications to the proof in order to replace 
this assumption with a much weaker assumption that can be thought of as 
a bit more restrictive version of the Kraichnan-type condition, and is consistent
with expected spatial complexity and uniform locality of fully developed turbulent flows.

The key new ingredients are redesigning the ensembles of test functions in the construction
of the ensemble averages in order to establish an explicit relation between the ensemble
averages at two different scales, and formulating uniform locality in
terms of Morrey-type quantities.

The main theorem states that the ensemble average of enstrophy fluxes within a range of scales is comparable to the modified mean palinstrophy (in particular it is positive).  In the first section the ensemble framework is presented.  Next, we state all of the assumptions on the flow used to prove
existence of the enstrophy cascade.  Then the proof of the main result is given.

\section{Ensemble Average Framework}
The inward enstrophy flux through a sphere is given by $-\int_{\partial B}\frac{1}{2}|\omega|^2 u\cdot n \,d\sigma$, where $n$ is the outward unit normal vector to the surface and the enstrophy is the $L^2$ norm of the vorticity $\omega= \text{curl } u$.  By Stokes' theorem, $-\int_{\partial B}\frac{1}{2}|\omega|^2 u\cdot n \,d\sigma=\int_{B} (u\cdot \nabla) \omega\cdot\omega \,dx$.  Instead of the sharp cutoff at the boundary of $B$, we need something smooth.  We will be working with the quantities $\int (u\cdot\nabla)\omega\cdot\psi\omega\,dx=-\int \frac{1}{2} |\omega|^2 u\cdot \nabla\psi \,dx$, where $\psi$ is a smooth function equal to 1 on $B$, supported on $2B$, and with inward pointing gradient, with certain bounds on its derivatives.

\begin{defi}[Refined test function]A $(C_0,\rho)$ test function at scale $R$ is any $\psi\in C^\infty(\mathbb{R}^3)$ supported in a ball of radius $2R$ with $0\leq\psi\leq 1$, $|\nabla \psi|<\frac{C_0}{R} \psi^\rho$, and $|\Delta\psi|<\frac{C_0}{R^2} \psi^{2\rho-1}$. \end{defi}

\begin{defi}[Ensemble] (Parameters $R, C_0, \rho, K_1, K_2, \psi_{0,R_0}$)
Fix a test function $\psi_0=\psi_{0,R_0}$.  The region of interest is $B(0,2R_0)$.

An ensemble is a collection of $(C_0,\rho)$ scale $R$ test functions $\{\psi_i\}_{i=1}^n$ satisfying the following properties:
\begin{enumerate}
\item $\psi_i \leq \psi_0 \leq \sum \psi_i$
\item $(R_0/R)^3\leq n\leq K_1 (R_0/R)^3$
\item No point of $B(2R_0,0)$ is contained in more than $K_2$ of the supports of $\psi_i$.
\end{enumerate}
\end{defi}
Choosing larger $K_1$, $K_2$ allows ensembles with greater global and local multiplicity respectively.

\begin{defi}[Ensemble average]
For a function $f$, denote by $\langle F \rangle _R$ the ensemble average $\frac{1}{n}\sum_{i=1}^n\frac{1}{R^3} \int f\psi_{i,R}\,dx$, and $F_0=\frac{1}{R_0^3}\int f \psi_0 \,dx$. 
\end{defi}
  Property 1 above is needed to compare $\langle F\rangle_R$ to $F_0$. Note that test functions near the boundary of the support of $\psi_0$ will have small integrals (due to Property 1), effectively skewing the ensemble average towards zero.  Larger $K_1, K_2$ allow ensembles that have higher weight on functions away from the boundary, making the skewing insignificant.

These ensemble averages can be viewed as a way to detect whether a function is significantly negative at some spatial scale.  If every ensemble average $\langle F \rangle _R$ (for fixed parameters) is positive, no matter how one arranges and stacks the test functions, then the function is not significantly negative at scale $R$.  Increasing $K_1$ and $K_2$ lowers the threshold for a function to be considered significantly negative.

Many ensembles can be constructed by applying Lemma 2 to $\psi_0$ and varying the multiplicity of the resulting functions (Assume $\psi_0$ satisfies the stronger $C_0'$-bounds to get an ensemble with $C_0$-bounds).

The following lemma states that ensemble averages (at any scale) of positive functions are comparable to the large scale mean.  The proof immediately follows from the definitions.

\begin{lem} If $f\geq 0$ then $\frac{1}{K_1}F_0 \leq \langle F\rangle_R \leq K_2 F_0$.  For slightly modified ensemble averages, we have $\frac{1}{n}\sum_1^n \frac{1}{R^3} \int f\psi_{i,R}^\delta\, dx \leq K_2\frac{1}{R_0^3} \int f\psi_0^\delta\, dx \;(\delta>0)$.\qed 
\end{lem}

Using a refined partition of unity, one can turn larger scale ensembles into smaller scale ensembles.

\begin{lem}
Any $(C_0,\rho)$ scale $R$ test function is a sum of $64(R/R')^3$ $(C_0',\rho)$ scale $R'$ test functions (where $C_0'$ depends only on $C_0$, $R>R'$).

Therefore for all $(K_1,K_2,C_0)$-ensembles at scale $R$ and every $R'<R$, there exists a $(64K_1,8K_2,C_0')$-ensemble at scale $R'$ such that $\langle F\rangle_R=\langle F \rangle_{R'}$.
\end{lem}
\begin{proof}
Let $\psi$ be a $(C_0,\rho)$ scale $R$ test function.  Now to construct the partition of unity, take a $(C_0,\rho)$, scale $R'$ 3D test function $g_0$, centered at zero and equal to 1 on $[-R,R]^3$ (such a function exists as long as $C_0$ isn't too small).  Define $g_p=g_0(x-2R'p)$, where $p\in \mathbb{Z}^3$.  Then $1\leq\sum_p g_p\leq 2$ so we may define $h_p= g_p / \sum_q g_q$.

Some calculus shows that $|\nabla h_p | < \frac{6C_0}{R'} h_p^\rho$ and $|\Delta h_p|<\frac{3C_0+10C_0^2}{R'^2} h_p^{2\rho-1}$, so $|\nabla (\psi h_p)|< \frac {7C_0}{R'} (\psi h_p)^\rho$ and $|\Delta (\psi h_p)|<\frac{4C_0+22C_0^2}{R'^2}(\psi h_p)^{2\rho-1}$. Fewer than $8\lceil R/R'\rceil^3\leq 64(R/R')^3$ of the functions $\psi h_p$ are nonzero, and for any $x$, $\psi_p(x)\neq 0$ for at most 8 functions.

Since $\psi = \sum_p \psi h_p$, we have the first claim.  For the second claim, given an ensemble $\{\psi_i\}_i$, the new ensemble will be $\{\psi_ih_p\}_{i,p}$.
\end{proof}

\section{Enstrophy Cascade}

\subsection{Assumptions}
Let $\Omega$ be a domain in $\mathbb{R}^3$, and  $u$  a Leray solution to the Navier-Stokes equations on $\Omega\times(0,\infty)$.  The enstrophy cascade will occur on a cylinder $B(0,2R_0)\times(0,T)$ where $B(0,2R_0+R_0^{2/3})\subset\Omega$, $R_0<1$ and $T>R_0^2$
($R_0$ will be the macro-scale in the problem).  Centering $B(0,2R_0)$ and starting the time interval at zero is for notational convenience.

It is required that the solution $u$ has $\sup_{t\in(0,T)} \int_{B(0,2R_0+R_0^{2/3})} |\omega(x,t)|\,dx <\infty$. If $\Omega=\mathbb{R}^3$ we require $\sup_{t\in(0,T)} \int_{\mathbb{R}^3} |\omega(x,t)|\,dx <\infty$.  This is guaranteed if $u$ is a classical solution on $(0,T)$ with $||\omega_0||_{L^1}<\infty$, or is a Leray solution with  finite Radon measure initial vorticity given by the retarded mollification method used in \cite{Constantin}.

Fix a $(C,\rho)$ scale $R_0$ test function $\psi_0$ for spatial localization.  Let a temporal cutoff function $\eta\in C^\infty[0,T]$ be such that $|\eta'|<\frac{C}{T}\eta^\rho$, $0\leq \eta\leq 1$, $\eta=0$ on $[0,T/3)$, and $\eta=1$ on $(2T/3,T]$, where $T>R_0^2$.  For any test function $\psi$, define $\phi(x,t)=\psi(x)\eta(t)$.

In \cite{G} it is shown that Assumption 1 implies that the localized enstrophy is bounded
on $(0,T)$.  Using the localized Biot-Savart law, this implies a bound on $||u(t)||_{H^1(B(R_0))}$. Then by the partial regularity theory of the Navier Stokes equations $u$ is smooth on $(0,T]\times B(0,R_0)$.

For simplicity we consider $\Omega=\mathbb{R}^3$ only.  If $\Omega$ is a bounded domain, the localized Biot-Savart law used in \cite{G} has a number of lower order terms, which will require a small modification of the assumptions. Namely, that $\sigma_0^{3/4}<\beta^{3/4} R_0$ in Assumption 2 and $R= (\sigma_0/\beta)^{3/4}$ in Theorem 2.

\begin{asu} Let $\xi=\omega/|\omega|$ be the vorticity direction field.  Assume there exist $M,C_1$ such that $|\sin \varphi(\xi(x,t),\xi(y,t))|\leq C_1|x-y|^{1/2}$ for a.e. $(x,y,t)$ in $(B(0,2R_0)\cap\{|\nabla u|>M\})\times B(0,2R_0+R_0^{2/3})\times(0,T)$, where $\varphi(z_1,z_2)$ denotes the angle between the vectors $z_1$ and $z_2$.
\end{asu}

This assumption is based on numerical simulations which suggest that regions of intense vorticity self organize into coherent vortex structures and in particular, vortex filaments.

\begin{asu}
Denote the scale-$R_0$ mean enstrophy by $E_0=\frac{1}{T}\int\frac{1}{R_0^3}\int\frac{1}{2}|\omega|^2\phi_0^{2\rho-1}dx\,dt$, the modified mean palinstrophy by $P_0=\frac{1}{T}\int\frac{1}{R_0^3}\int|\nabla\omega|^2\phi_0\,dx\,dt+\frac{1}{T}\frac{1}{R_0^3}\int|\omega(x,T)|^2\psi_0\,dx$, and the modified Kraichnan scale by $\sigma_0=(\frac{E_0}{P_0})^{1/2}$.  It is required that $\sigma_0<\beta R_0$ ($0<\beta<1$ is a constant depending only on $C_0, C_1, M, K_1, K_2,$ and $B_T:=\sup_{t\in(0,T)}||\omega||_{L^1(\Omega)}$, and $\beta$ shrinks to zero as any of them increase to infinity).
\end{asu}

The second term of the modified palinstrophy arises from the shape of the temporal cutoff $\eta$.

Denote by $M^{p,q}=M^{p,q}(B(0,2R_0+R_0^{2/3}))$ the Morrey space of functions $f$ such that $\sup_{y,R}\frac{1}{R^{3(1-p/q)}}\int_{B(y,R)\cap B(0,2R_0+R_0^{2/3})}|f|^p\,dx$ is finite.  Note that $L^q\subset M^{p,q}\subset L^p $.

\begin{asu}
Assume $\omega(t,x)\in L^2(0,T; M^{2,q})$ with $\sigma_0^{1-2/q}||\omega||_{L^2_tM^{2,q}_x}<(\frac{\beta}{2})^{1-2/q} \frac{1}{C}$ where $C$ depends only on $\beta, C_0, K_1, K_2$.
\end{asu}

Assumption 3 will be used with the bound $||\omega||_{L^2((0,T)\times B(x_i,2R+R^{2/3}))}\leq cR^{1-2/q}||\omega||_{L^2_tM^{2,q}_x}$.
All Leray solutions have $||\omega||_{L^2_tL^2_x}<\infty$.  We need slightly more -- that $||\omega||_{L^2_tM^{2,q}_x}<\infty$. 
Assumptions 2 and 3 will be true if $||\omega||_{L^2_t M^{2,q}_x}$ is sufficiently small relative to $P_0$, that is, we consider high (time-averaged) spatial complexity of $\omega$ in $B(0,2R_0)$.

The final assumption used to prove the enstrophy cascade is that the enstrophy doesn't drop off too much at time $T$.

\begin{asu}[Modulation]$\int|\omega(x,T)|^2\psi_0(x)\,dx\geq \frac{1}{2} \sup_{t\in (0,T)}\int|\omega(x,t)|^2\psi_0(x)\,dx$.
\end{asu}

\subsection{Theorems}

To work with the enstrophy we use the vorticity form of the Navier-Stokes equations: $$\omega_t+(u\cdot\nabla)\omega = (\omega\cdot\nabla)u + \Delta\omega.$$

To use the ensemble average framework, we take as our function $f=-\frac{1}{T}\int_0^T (u\cdot\nabla)\omega\cdot\omega \,\eta \,dt$ so that $\int f\psi_i\,dx = \frac{1}{T}\int_0^T\int\frac{1}{2}|\omega|^2(u\cdot\nabla\phi_i)\,dx\,dt$.  This time averaged enstrophy flux along $\nabla\phi_i$ will represent the amount of enstrophy flowing into the scale $R$ if $\phi_i$ is taken to be constant on $B(x_i,R)$ and with inward pointing gradient.

The following lemmas will be used in the proof of the enstrophy cascade, and are essentially as they appear in \cite{DG1,DG2}.

\begin{lem}For a solution to the Navier-Stokes equations $u$ that is smooth on $[0,T]\times B(x_i,2R)$,  
\begin{align*}\int_0^T\int\frac{1}{2}|\omega|^2(u\cdot \nabla\phi_i)\,dx\,dt=&\int\frac{1}{2}|\omega(x,T)|^2\psi_i(x)\,dx+\int_0^T\int|\nabla\omega|^2\phi_i\,dx\,dt\\&-\int_0^T\int\frac{1}{2}|\omega|^2(\partial_t\phi_i+\Delta\phi_i)\,dx\,dt-\int_0^T\int(\omega\cdot\nabla)u\cdot\phi_i\omega\,dx\,dt.\end{align*}
\end{lem}

\begin{proof}  Using integration by parts and that $u$ is divergence free, $\frac{1}{2}|\omega|^2(u\cdot\nabla\phi_i) = - (u\cdot\nabla)\omega\cdot\phi_i\omega$.
From the Navier-Stokes equations, $-(u\cdot\nabla)\omega = \partial_t\omega -\Delta\omega - (\omega\cdot\nabla)u$.  Integrating in space and time against $\phi_i\omega$ yields $-\int_0^T\int(u\cdot\nabla)\omega\cdot \phi_i\omega \,dx\,dt = \int_0^T\int\partial_t\omega\cdot \phi_i\omega \,dx\,dt -\int_0^T\int\Delta\omega\cdot \phi_i\omega \,dx\,dt - \int_0^T\int(\omega\cdot\nabla)u\cdot \phi_i\omega \,dx\,dt$.  Now 
\begin{align*}
\int_0^T\int\partial_t\omega\cdot \phi_i\omega \,dx\,dt=\int_0^T \int\frac{1}{2}\partial_t(\omega\cdot \phi_i\omega)-\frac{1}{2}\omega\cdot(\partial_t\phi_i)\omega \,dx\,dt\\=\int|\omega(T)|^2 \psi_i \,dx - \int_0^T\int\frac{1}{2}|\omega|^2\partial_t\phi_i\,dx\,dt.
\end{align*}  and $$-\int_0^T\int\Delta\omega\cdot\phi_i\omega\,dx\,dt = \int_0^T\int |\nabla\omega|^2\phi_i\,dx\,dt -\int_0^T\int \frac{1}{2} |\omega|^2 \Delta\phi_i \,dx\,dt,$$
simply by integration by parts and the fundamental theorem of calculus.
Putting all of the equations together completes the proof.

\end{proof}

\begin{lem}  For a divergence free function $u\in H^1(\mathbb{R}^3)^3$, with $\omega:=\nabla\times u$, we have $(\omega\cdot\nabla)u\cdot
\omega(x)=c\,P.V.\int \omega(x)\times\omega(y) \cdot G_\omega(x,y)
\,dy$ for a.e. $x$, where $(G_\omega(x,y))_k=
\frac{\partial^2}{\partial x_i\partial y_k}
\frac{1}{|x-y|} \omega_i(x)$
$=\frac{(x_k-y_k)(x_i-y_i)}{|x-y|^5}\omega_i(x)+\frac{1}{|x-y|^3}\omega_k(x)$.\end{lem}

\begin{proof}
For any divergence free Schwartz function $u$, we have $\Delta u = -$curl $\omega$, so $u=c\,\frac{1}{|\cdot|} \,*\, $curl $\omega$, and $\partial_i u^j = c\,\partial_i\frac{1}{|\cdot|} \,*\, \epsilon_{jkl}\partial_k\omega^l = c\,\partial_i\partial_k\frac{1}{|\cdot|} \,*\, \epsilon_{jkl}\omega^l $.  Then $$\omega^j\partial_j u^i\omega^i (x) = c\, \omega^j\omega^i (\partial_i\partial_k\frac{1}{|\cdot|} \,*\, \epsilon_{jkl}\omega^l)(x)= c\,P.V.\int \omega(x)\times\omega(y) \cdot G_\omega(x,y)\,dy.$$

By density of Schwartz functions in $H^1$, the following holds in $L^2$:
$$(\omega\cdot\nabla)u\cdot \omega = c P.V.\int \omega(x)\times\omega(y)\cdot G_\omega(x,y)\,dy.$$
\end{proof}

Now we wish to bound the vortex stretching term by integrals of positive functions in order to use Lemma 1.

\begin{lem}
\begin{align*}
&\Big|\int_0^T\int (\omega\cdot \nabla)u\cdot\phi_i\omega\,dx\,dt\Big|\leq c||\omega||_*\Big( \sup_t \int\frac{1}{2}|\omega(x,t)|^2\psi_i(x)\,dx + \int_0^T|\nabla\omega|^2\phi_i\,dx\,dt\Big)
+\\& \frac{c'+c''||\omega||_*}{R^2}\int_0^T\int \frac{1}{2}|\omega|^2 \phi_i^{2\rho-1} \,dx\,dt
\end{align*}
$(||\omega||_*:=||\omega||_{L^2(B(x_i,R)\times(0,T))}$, constants depend on $M, B_T, C_0$).

\end{lem}
\begin{proof}
$\int_0^T\int (\omega\cdot \nabla)u\cdot\phi_i\omega\,dx\,dt= \int_0^T\int_{\{|\nabla u|<M\}}(\omega\cdot\nabla)u\cdot\phi_i\omega\,dx\,dt+ \int_0^T\int_{\{|\nabla u|>M\}}(\omega\cdot\nabla)u\cdot\phi_i\omega\,dx\,dt$.  The first term is easily bounded by $\frac{M}{R^2}\int_0^T\int |\omega|^2\phi_i^{2\rho-1}$.  For the second term, the Biot-Savart law gives us 
\begin{align}
&\int_0^T\int_{\{|\nabla u|>M\}}(\omega\cdot\nabla)u\cdot\phi_i\omega\,dx\,dt=\\&\int_0^T\int_{\{|\nabla u|>M\}}P.V.\int \omega(x)\times \omega(y)\cdot G_\omega(x,y)\phi_i(x)\,dy\,dx\,dt=\\&\int_0^T\int_{\{|\nabla u|>M\}}P.V.\int_{\{|x-y|<R^{2/3}\}}\omega(x)\times \omega(y)\cdot G_\omega(x,y)\phi_i(x)\,dy\,dx\,dt \\&+\int_0^T\int_{\{|\nabla u|>M\}}\int_{\{|x-y|>R^{2/3}\}}\omega(x)\times \omega(y)\cdot G_\omega(x,y)\phi_i(x)\,dy\,dx\,dt.
\end{align}

The second term (4) is bounded by \begin{align*}
\frac{1}{R^2}\int_0^T\int\int_{\{|x-y|>R^{2/3}\}}|\omega(x)|^2|\omega(y)|\phi_i(x)\,dy\,dx\,dt\leq \frac{1}{R^2} \sup_t ||\omega(t)||_{L^1} \int_0^T\int |\omega|^2\phi_i^{2\rho-1}\,dx\,dt.
\end{align*}

For the first term (3), since $$\big|\omega(x)\times\omega(y)\cdot G_\omega(x,y)\big|\leq|\omega(x)||\omega(y)||\sin\varphi(\,\xi(x),\xi(y)\,)||G_\omega(x,y)|\leq \frac{|\omega(x)|^2|\omega(y)|}{|x-y|^{5/2}},$$ 
we have
\begin{align}
&\int_0^T\int_{\{|\nabla u|>M\}}\Big|P.V.\int_{\{|x-y|<R^{2/3}\}}\omega(x)\times\omega(y)\cdot G_\omega(x,y)\,dy\Big|\phi_i(x)\,dx\,dt \leq
\\&\int_0^T\int_{\{|\nabla u|>M\}}\int_{\{|x-y|<R^{2/3}\}}\frac{|\omega(y)||\omega(x)|^2}{|x-y|^{5/2}}\phi_i(x)\,dy\,dx\,dt\leq 
\\& c \int_0^T||\omega||_{L^2(B(x_i,2R+R^{2/3})}||\,|\phi_i^{1/2}\omega|^2\,||_{3/2}dt\leq
\\& c\int_0^T||\omega||_{L^2(B(x_i,2R+R^{2/3}))}||\phi_i^{1/2}\omega||_2||\nabla(\phi_i^{1/2}\omega)||_2dt\leq
\\& c ||\omega||_* \sup_t||\psi_i^{1/2}\omega||_2\Big(\int_0^T||\nabla(\phi_i^{1/2}\omega)||_2^2\,dt\Big)^{1/2}\leq
\\& c||\omega||_* \Big(\frac{1}{2}\sup_t||\psi_i^{1/2}\omega||_2^2 + \int_0^T||\nabla(\phi_i^{1/2}\omega)||_2^2\,dt\Big) \leq
\\& c||\omega||_*\Big(\frac{1}{2} \sup_t||\psi_i^{1/2}\omega||_2^2 + 2\int_0^T\int |\nabla\omega|^2\phi_i\,dx\,dt +\frac{c}{2R^2}\int_0^T\int |\omega|^2\phi_i^{2\rho-1}\,dx\,dt\Big),
\end{align}
 using $|\nabla(\phi_i^{1/2}\omega)|^2\leq 2|\nabla\omega|^2\phi_i + \frac{1}{2} \frac{|\nabla\phi_i|^2}{\phi_i} |\omega|^2\leq 2|\nabla\omega|^2\phi_i + \frac{c}{2R^2}|\omega|^2\phi_i^{2\rho-1}$ for the last inequality (11), and Hardy-Littlewood-Sobolev to reach (7), Gagliardo-Nirenberg to reach (8), and Cauchy-Schwarz to reach (9).  Collecting the bounds on the various terms proves the lemma.

\end{proof}

\begin{thm} $\frac{1}{4K_1}P_0\leq \langle F\rangle_R\leq (K_2+\frac{1}{4K_1}) P_0$ for $R=\sigma_0/\beta$, and $(K_1,K_2)$-ensemble averages.
\end{thm}

\begin{proof}
For an individual test function we have
\begin{align*}
&F_i:=\int_0^T\int\frac{1}{2}|\omega|^2(u\cdot \nabla\phi_i)\,dx\,dt=\Big(\sup_t\int\frac{1}{2}|\omega(x,t)|^2\psi_i(x)\,dx+\int_0^T\int|\nabla\omega|^2\phi_i\,dx\,dt\Big)-
\\&\int_0^T\int\frac{1}{2}|\omega|^2(\partial_t\phi_i+\Delta\phi_i)\,dx\,dt-\int_0^T\int(\omega\cdot\nabla)u\cdot\phi_i\omega\,dx\,dt =: A_i - B_i - C_i. \end{align*}

Using Assumption 4 and Lemma 1, $\frac{1}{2K_1} P_0\leq \frac{1}{nTR^3}\sum_1^n A_i \leq K_2 P_0$. Next, $|B_i|\leq \frac{c}{R^2}\int_0^T\int \frac{1}{2}|\omega|^2 \phi_i^{2\rho-1} \,dx\,dt$ so $|\frac{1}{nTR^3}\sum_1^n B_i| \leq \frac{cK_2}{R^2}E_0\leq cK_2\beta^2 P_0 \leq \frac{1}{8K_1} P_0$ for an appropriate choice of $\beta$.

Using the vortex stretching term lemma and Assumption 3, 
\begin{align*}
&\big|\frac{1}{nTR^3}\sum_1^n C_i\big|\leq (c+c'||\omega||)\frac{K_2}{R^2}E_0 +  c''||\omega||K_2P_0\leq 
\\&(c+c' R^{1-2/q}||\omega||_{L^2_tM^{2,q}_x})\frac{K_2}{R^2}E_0 + c''R^{1-2/q}||\omega||_{L^2_tM^{2,q}_x}K_2P_0< \frac{1}{8K_1} P_0 \end{align*}

Then $\langle F\rangle_R=\frac{1}{nTR^3}\sum_1^n F_i = \frac{1}{nTR^3}\sum_1^n A_i -\frac{1}{nTR^3}\sum_1^n B_i-\frac{1}{nTR^3}\sum_1^n C_i$ hence $\frac{1}{4K_1}P_0 \leq\langle F \rangle_R\leq (K_2+\frac{1}{4K_1}) P_0$

\end{proof}

\begin{thm}$\frac{1}{256K_1}P_0\leq \langle F\rangle_R\leq (8K_2+\frac{1}{256K_1})P_0$ for $(C_0, \rho, K_1,K_2)$-ensembles with $\sigma_0/\beta\leq R \leq R_0$ .
\end{thm}

\begin{proof}  By Lemma 2, every $(C_0,\rho,K_1,K_2)$ scale $R$ ensemble average is equal to some $(C_0', \rho,\linebreak 64K_1, 8K_2)$ scale $\sigma_0/\beta$ ensemble average, which satisfies the desired inequalities.
\end{proof}
In particular this holds for ensembles of test functions with inward pointing gradient that are constant on $B(x_i,R)$.  Thus in this precise sense solutions of the Navier-Stokes equations that satisfy the assumptions exhibit an enstrophy cascade. 
 
\section{Conclusion}

Fix $R_0, T, K_1, K_2, C_0, C_1, M, B_T,\text{and } q$.  There will be an enstrophy cascade on $B(0,2R_0)\times(0,T)$ for any Leray solution with $E_0/P_0$ and $||\omega||_{L_t^2M^{2,q}_x}/P_0$ sufficiently small, as long as there is coherence of vorticity direction where $|\nabla u|>M$ and 
the macro-scale enstrophy before time $T$ never exceeds twice the local enstrophy at time $T$. 
The assumptions are consistent
with strong local anisotropy, spatial complexity and uniform locality of fully
developed 3D turbulent flows.

\section*{Acknowledgements}

I would like to thank my advisor, Zoran Grujic, for suggesting the problem.

\bibliographystyle{plain}

\bibliography{bib}

\end{document}